\documentclass{amsart}
\usepackage{amssymb}
\usepackage{amsmath}
\usepackage{amsfonts}

\newtheorem{theorem}{Theorem}
\theoremstyle{plain}

\newtheorem{corollary}{Corollary}

\newtheorem{definition}{Definition}

\newtheorem{proposition}{Proposition}
\newtheorem{remark}{Remark}

\numberwithin{equation}{section}
\input{tcilatex}

\begin{document}
\title[The Kolmogorov-Riesz Compactness Theorem and Some Compactness
Criterions]{The Kolmogorov-Riesz Theorem and Some Compactness Criterions of
Bounded Subsets in Weighted Variable Exponent Amalgam and Sobolev Spaces}
\author{Ismail AYDIN}
\address{Sinop University Faculty of Arts and Sciences Department of
Mathematics \\
Sinop, TURKEY}
\email{iaydin@sinop.edu.tr}
\urladdr{}
\thanks{}
\author{Cihan UNAL}
\address{Sinop University Faculty of Arts and Sciences Department of
Mathematics \\
Sinop, TURKEY}
\email{cunal@sinop.edu.tr}
\urladdr{}
\date{}
\dedicatory{}
\thanks{}

\begin{abstract}
We study totally bounded subsets in weighted variable exponent amalgam and
Sobolev spaces. Moreover, this paper includes several detailed generalized
results of some compactness criterions in these spaces.
\end{abstract}

\maketitle

\section{Introduction}

Initially, the classical Riesz-Kolmogorov theorem states about the
compactness of subsets in $L^{p}\left[ 0,1\right] $ for $1<p<\infty $, see 
\cite{Kol}. This theorem has been generalized to some function spaces, such
as Takahashi \cite{Tak} for Orlicz spaces, Goes and Welland \cite{Go} for K%
\"{o}the spaces, Musielak \cite{Mus} for Musielak-Orlicz spaces, Rafeiro 
\cite{Raf} for variable exponent Lebesgue spaces, Bandaliyev \cite{Ba} for
weighted variable exponent Lebesgue spaces, G\'{o}rka and Rafeiro \cite{Gr}
for more general framework, namely in the case of Banach function spaces
(shortly, BF-space) and grand variable variable Lebesgue spaces, G\'{o}rka
and Macios \cite{Gor}, \cite{Gom} for classical Lebesgue variable exponent
Lebesgue in metric measure spaces. Weil \cite{Wei} considered compactness in 
$L^{p}$- spaces on locally compact groups. Moreover, the compactness
problems for various spaces of differentiable functions on the Euclidean
spaces have been studied by several authors. The classical criterion of
Kolmogorov-Riesz for compactness of subsets of $L^{p}\left( 1\leq p<\infty
\right) $ has been extended by Feichtinger \cite{Fie} to translation
invariant Banach function spaces. More details can be seen \cite{Hoh} and 
\cite{Pan}.

The amalgam of $L^{p}$ and $l^{q}$ on the real line is the space $\left(
L^{p},l^{q}\right) \left( 
\mathbb{R}
\right) $ (briefly, $\left( L^{p},l^{q}\right) $) consisting of functions
which are locally in $L^{p\text{ }}$ and have $l^{q}$ behavior at infinity.
Wiener \cite{Wie} studied several special cases of amalgam spaces including $%
\left( L^{1},l^{2}\right) $, $\left( L^{2},l^{\infty }\right) $, $\left(
L^{\infty },l^{1}\right) $ and $\left( L^{1},l^{\infty }\right) $.
Comprehensive information about amalgam spaces can be found in \cite{Fou}, 
\cite{Ho} and \cite{Sq}. Recently, there have been many interesting and
important papers appeared in variable exponent amalgam spaces $\left(
L^{r(.)},\ell ^{s}\right) $ such as Aydin \cite{Ay}, Aydin \cite{Ady}, Aydin
and Gurkanli\ \cite{Ag}, Gurkanli \cite{Gur}, Gurkanli and Aydin \cite{Gra},
Hanche-Olsen and Holden \cite{Hoh}, Meskhi and Zaighum \cite{MesZa},
Kokilashvili, Meskhi and Zaighum \cite{Kmz} and Kulak and Gurkanli \cite{Kul}%
. In 2003, Pandey studied the compactness of bounded subsets in a Wiener
amalgam space $W(B,Y)$ whose local and global components are solid Banach
function spaces and satisfy conditions in \cite[Definition 5.1]{Pan}.

In this study, we focus especially on the spaces $\left( L_{w}^{p(.)},\ell
^{q}\right) $ on $%
\mathbb{R}
$, and discuss totally bounded subsets in weighted variable exponent amalgam
and Sobolev spaces. Moreover, it is well known that the spaces $L_{w}^{p(.)}$
and $\left( L_{w}^{p(.)},\ell ^{q}\right) $ are not translation invariant
and the map $y\longrightarrow L_{y}f$ is not continuous for $f\in \left(
L_{w}^{p(.)},\ell ^{q}\right) $. Also, the Young theorem $\left\Vert f\ast
g\right\Vert _{p(.),w}\leq \left\Vert f\right\Vert _{p(.),w}\left\Vert
g\right\Vert _{1}$ is not valid for $f\in L_{w}^{p(.)}\left( 
\mathbb{R}
^{n}\right) $ and $g\in L^{1}\left( 
\mathbb{R}
^{n}\right) ,$ see \cite{Kor}. Hence, our concerned amalgam space $\left(
L_{w}^{p(.)},\ell ^{q}\right) $ whose local component doesn't provide
conditions in \cite[Definition 5.1]{Pan}. In this regard, we will propose
some new conditions and criterions for compactness of bounded subsets in
these spaces. In addition, Gurkanli \cite{Gur} showed that $\left(
L_{w}^{p(.)},\ell ^{q}\right) =L_{w}^{p(.)}$ under the some conditions, that
is, the space $\left( L_{w}^{p(.)},\ell ^{q}\right) $ is a extension of $%
L_{w}^{p(.)}$. Finally, we will obtain that our main theorem provides a
generalization of the corresponding results, such as Bandaliyev \cite{Ba},
Bandaliyev and G\'{o}rka \cite{BanGor}, G\'{o}rka and Macios \cite{Gor}, 
\cite{Gom}, G\'{o}rka and Rafeiro \cite{Gr} and Rafeiro \cite{Raf}.

\section{\textbf{Notations and Preliminaries}}

In this section, we give some essential definitions, theorems and
compactness criterions about totally bounded subsets in weighted variable
exponent Lebesgue spaces.

\begin{definition}
Let $\left( X,d\right) $ be a metric space and $\varepsilon >0$. A subset $K$
in $X$ is called a $\varepsilon $-net (or $\varepsilon $-cover) for $X$ if
for every $x\in X$ there is a $x_{\varepsilon }\in K$ such that $d\left(
x,x_{\varepsilon }\right) <\varepsilon $. Moreover, a metric space is called
totally bounded if it admits a finite $\varepsilon $-net.
\end{definition}

It is known that a subset of a complete metric space relatively compact
(i.e. its closure is compact) if and only if it is totally bounded, see \cite%
{Yo}.

\begin{theorem}
Assume that X is a metric space and $K\subset X.$ Then the following
conditions are equivalent.

\begin{enumerate}
\item[\textit{(i)}] $K$ is totally bounded (or, precompact) and complete

\item[\textit{(ii)}] $K$ is compact.
\end{enumerate}
\end{theorem}

\begin{definition}
Let $X$ and $Y$ be metric spaces. A family $\tciFourier $ of functions from $%
X$ to $Y$ is said to be equicontinuous if given $\varepsilon >0$ there
exists a number $\delta >0$ such that $d_{Y}\left( f(x),f(y)\right)
<\varepsilon $ for all $f\in \tciFourier $ and all $x,y\in X$ satisfying $%
d_{X}\left( x,y\right) <\delta .$
\end{definition}

The following theorem is quite useful for several compactness results.

\begin{theorem}
(\cite{Hoh})Let $X$ be a metric space. Assume that, for every $\varepsilon
>0,$ there exists some $\delta >0$, a metric space $W$, and a mapping $\Phi
:X\longrightarrow W$ so that $\Phi \left[ X\right] $ is totally bounded, and
whenever $x,y\in X$ are such that $d\left( \Phi (x),\Phi (y)\right) <\delta $%
, then $d\left( x,y\right) <\varepsilon .$ Then $X$ is totally bounded.
\end{theorem}

\begin{definition}
(\cite{Kor})For a measurable function $p\left( .\right) :%
\mathbb{R}
^{n}\longrightarrow \left[ 1,\infty \right) $ (called the variable exponent
on $%
\mathbb{R}
^{n}$ by the symbol $P\left( 
\mathbb{R}
^{n}\right) $), we put 
\begin{equation*}
p^{-}=\underset{x\in 
\mathbb{R}
^{n}}{\text{essinf}}p(x)\text{, \ \ \ \ \ \ }p^{+}=\underset{x\in 
\mathbb{R}
^{n}}{\text{esssup}}p(x)\text{.}
\end{equation*}%
The variable exponent Lebesgue spaces $L^{p(.)}(%
\mathbb{R}
^{n})$ is defined as the set of all measurable functions $f$ on $%
\mathbb{R}
^{n}$ such that $\varrho _{p(.)}(\lambda f)<\infty $ for some $\lambda >0$,
equipped with the Luxemburg norm%
\begin{equation*}
\left\Vert f\right\Vert _{p(.)}=\inf \left\{ \lambda >0:\varrho _{p\left(
.\right) }\left( \frac{f}{\lambda }\right) \leq 1\right\} \text{,}
\end{equation*}%
where $\varrho _{p(.)}(f)=\dint\limits_{%
\mathbb{R}
^{n}}\left\vert f(x)\right\vert ^{p(x)}dx.$

If $p^{+}<\infty $, then $f\in L^{p(.)}(%
\mathbb{R}
^{n})$ iff $\varrho _{p(.)}(f)<\infty $. The set $L^{p(.)}(%
\mathbb{R}
^{n})$ is a Banach space with the norm $\left\Vert .\right\Vert _{p(.)}$.
Moreover, the norm $\left\Vert .\right\Vert _{p(.)}$ coincides with the
usual Lebesgue norm $\left\Vert .\right\Vert _{p}$ whenever $p(.)=p$ is a
constant function.
\end{definition}

\begin{definition}
A measurable and locally integrable function $w:%
\mathbb{R}
^{n}\longrightarrow \left( 0,\infty \right) $ is called a weight function.
The weighted modular is defined by%
\begin{equation*}
\varrho _{p(.),w}(f)=\dint\limits_{%
\mathbb{R}
^{n}}\left\vert f(x)\right\vert ^{p(x)}w(x)dx.
\end{equation*}%
The weighted variable exponent Lebesgue space $L_{w}^{p(.)}(%
\mathbb{R}
^{n})$ consists of all measurable functions $f$ on $%
\mathbb{R}
^{n}$ for which $\left\Vert f\right\Vert _{L_{w}^{p(.)}(%
\mathbb{R}
^{n})}=\left\Vert fw^{\frac{1}{p(.)}}\right\Vert _{p(.)}<\infty $. Also, $%
L_{w}^{p(.)}(%
\mathbb{R}
^{n})$ is a uniformly convex Banach space, thus reflexive.

The relations between the modular $\varrho _{p(.),w}(.)$ and $\left\Vert
.\right\Vert _{L_{w}^{p(.)}(%
\mathbb{R}
^{n})}$ as follows%
\begin{eqnarray*}
\min \left\{ \varrho _{p(.),w}(f)^{\frac{1}{p^{-}}},\varrho _{p(.),w}(f)^{%
\frac{1}{p^{+}}}\right\} &\leq &\left\Vert f\right\Vert _{L_{w}^{p(.)}(%
\mathbb{R}
^{n})}\leq \max \left\{ \varrho _{p(.),w}(f)^{\frac{1}{p^{-}}},\varrho
_{p(.),w}(f)^{\frac{1}{p^{+}}}\right\} \\
\min \left\{ \left\Vert f\right\Vert _{L_{w}^{p(.)}(%
\mathbb{R}
^{n})}^{p^{+}},\left\Vert f\right\Vert _{L_{w}^{p(.)}(%
\mathbb{R}
^{n})}^{p^{-}}\right\} &\leq &\varrho _{p(.),w}(f)\leq \max \left\{
\left\Vert f\right\Vert _{L_{w}^{p(.)}(%
\mathbb{R}
^{n})}^{p^{+}},\left\Vert f\right\Vert _{L_{w}^{p(.)}(%
\mathbb{R}
^{n})}^{p^{-}}\right\} .
\end{eqnarray*}%
Moreover, if $0<C\leq w$, then we have $L_{w}^{p(.)}(%
\mathbb{R}
^{n})\hookrightarrow L^{p(.)}(%
\mathbb{R}
^{n}),$ since one easily sees that 
\begin{equation*}
C\dint\limits_{%
\mathbb{R}
^{n}}\left\vert f(x)\right\vert ^{p(x)}dx\leq \dint\limits_{%
\mathbb{R}
^{n}}\left\vert f(x)\right\vert ^{p(x)}w(x)dx
\end{equation*}%
and $C\left\Vert f\right\Vert _{p(.)}\leq \left\Vert f\right\Vert
_{L_{w}^{p(.)}(%
\mathbb{R}
^{n})}$, see \cite{Ayd}.
\end{definition}

\begin{theorem}
Let $p\left( .\right) ,q\left( .\right) \in P\left( 
\mathbb{R}
^{n}\right) $ such that $\frac{1}{p(.)}+\frac{1}{q(.)}=1$. Then for $f\in
L_{w}^{p(.)}(%
\mathbb{R}
^{n})$ and $g\in L_{w^{\ast }}^{q(.)}(%
\mathbb{R}
^{n})$, we have $fg\in L^{1}(%
\mathbb{R}
^{n})$ and%
\begin{equation*}
\dint\limits_{%
\mathbb{R}
^{n}}\left\vert f(x)g(x)\right\vert dx\leq C\left\Vert f\right\Vert
_{L_{w}^{p(.)}(%
\mathbb{R}
^{n})}\left\Vert g\right\Vert _{L_{w^{\ast }}^{q(.)}(%
\mathbb{R}
^{n})},
\end{equation*}%
where $w^{\ast }=w^{1-q(.)}.$
\end{theorem}

\begin{proof}
By the H\"{o}lder inequality for variable exponent Lebesgue spaces, we get 
\begin{eqnarray*}
\dint\limits_{%
\mathbb{R}
^{n}}\left\vert f(x)g(x)\right\vert dx &=&\dint\limits_{%
\mathbb{R}
^{n}}\left\vert f(x)g(x)\right\vert w(x)^{\frac{1}{p(x)}-\frac{1}{p(x)}}dx \\
&\leq &C\left\Vert fw^{\frac{1}{p(.)}}\right\Vert _{p(.)}\left\Vert gw^{-%
\frac{1}{p(.)}}\right\Vert _{q(.)}
\end{eqnarray*}%
for some $C>0$. That is the desired result.
\end{proof}

The space $L_{loc}^{1}\left( 
\mathbb{R}
^{n}\right) $ consists of all measurable functions $f$ on $%
\mathbb{R}
^{n}$ such that $f\chi _{K}\in L^{1}\left( 
\mathbb{R}
^{n}\right) $ for any compact subset $K\subset 
\mathbb{R}
^{n}$. It is a topological vector space with the family of seminorms $%
f\longrightarrow \left\Vert f\chi _{K}\right\Vert _{L^{1}}$. A Banach
function space (shortly, BF-space) on $%
\mathbb{R}
^{n}$ is a Banach space $\left( B,\left\Vert .\right\Vert _{B}\right) $ of
measurable functions which is continuously embedded into $L_{loc}^{1}\left( 
\mathbb{R}
^{n}\right) $, briefly $B\hookrightarrow L_{loc}^{1}\left( 
\mathbb{R}
^{n}\right) ,$ i.e. for any compact subset $K\subset 
\mathbb{R}
^{n}$ there exists some constant $C_{K}>0$ such that $\left\Vert f\chi
_{K}\right\Vert _{L^{1}}\leq C_{K}\left\Vert f\right\Vert _{B}$ for all $%
f\in B$.

The dual space of $L_{w}^{p(.)}(%
\mathbb{R}
^{n})$ is $L_{w^{\ast }}^{q(.)}(%
\mathbb{R}
^{n})$, where $\frac{1}{p(.)}+\frac{1}{q(.)}=1$ and $w^{\ast }=w^{1-q(.)}$,
see \cite{Lah}. Also, it is known that if $X$ is a BF-space, then the dual
space $X^{\ast }$ consisting of $g$ such that 
\begin{equation*}
\left\Vert g\right\Vert _{X^{\ast }}=\sup_{\substack{ f\in X  \\ \left\Vert
f\right\Vert _{X}\leq 1}}\dint\limits_{%
\mathbb{R}
^{n}}\left\vert g(x)f(x)\right\vert dx
\end{equation*}%
is also a BF-space. If using H\"{o}lder inequality for variable Lebesgue
spaces, then we have%
\begin{equation*}
\dint\limits_{%
\mathbb{R}
^{n}}\left\vert f(x)g(x)\right\vert dx\leq C\left\Vert f\right\Vert
_{L_{w}^{p(.)}(%
\mathbb{R}
^{n})}\left\Vert g\right\Vert _{L_{w^{\ast }}^{q(.)}(%
\mathbb{R}
^{n})}
\end{equation*}%
and%
\begin{equation*}
\left\Vert g\right\Vert _{\left( L_{w}^{p(.)}(%
\mathbb{R}
^{n})\right) ^{\ast }}\leq \left\Vert g\right\Vert _{L_{w^{\ast }}^{q(.)}(%
\mathbb{R}
^{n})}
\end{equation*}%
for some $C>0$. Therefore, the norm $\left\Vert .\right\Vert _{\left(
L_{w}^{p(.)}(%
\mathbb{R}
^{n})\right) ^{\ast }}$ is well defined. Moreover, $\left\Vert .\right\Vert
_{\left( L_{w}^{p(.)}(%
\mathbb{R}
^{n})\right) ^{\ast }}$ and $\left\Vert .\right\Vert _{L_{w^{\ast }}^{q(.)}(%
\mathbb{R}
^{n})}$ are equivalent by the similar methods for dual spaces of $L^{p(.)}$,
see \cite{Kor}. Therefore, there is a isometric isomorphism between $\left(
L_{w}^{p(.)}(%
\mathbb{R}
^{n})\right) ^{\ast }$ and $L_{w^{\ast }}^{q(.)}(%
\mathbb{R}
^{n}).$ This yields that $\left( L_{w}^{p(.)}(%
\mathbb{R}
^{n})\right) ^{\ast }=L_{w^{\ast }}^{q(.)}(%
\mathbb{R}
^{n}).$

\begin{remark}
Let $K\subset 
\mathbb{R}
^{n}$ with $\left\vert K\right\vert <\infty $. Then we have $\left\Vert \chi
_{K}\right\Vert _{L_{w}^{p(.)}(%
\mathbb{R}
^{n})}<\infty $.
\end{remark}

\begin{proof}
Fix $\lambda \geq 1$. Since weight function $w$ is locally integrable
function $%
\mathbb{R}
^{n}$, then 
\begin{eqnarray*}
\varrho _{p(.),w}\left( \frac{\chi _{K}}{\lambda }\right) &=&\dint\limits_{%
\mathbb{R}
^{n}}\frac{\left\vert \chi _{K}(x)\right\vert ^{p(x)}w(x)}{\lambda ^{p(x)}}%
dx=\dint\limits_{K}\lambda ^{-p(x)}w(x)dx \\
&\leq &\lambda ^{-p^{-}}\dint\limits_{K}w(x)dx\leq \lambda ^{-1}C_{K},
\end{eqnarray*}%
where $C_{K}=\dint\limits_{K}w(x)dx<\infty $. If we take $\lambda =C_{K}+1>0$%
, then we have $\left\Vert \chi _{K}\right\Vert _{L_{w}^{p(.)}(%
\mathbb{R}
^{n})}\leq C_{K}+1$.
\end{proof}

\begin{definition}
For $x\in 
\mathbb{R}
^{n}$ and $r>0,$ we denote an open ball with center $x$ and radius $r$ by $%
B(x,r)$. For $f\in L_{loc}^{1}\left( 
\mathbb{R}
^{n}\right) ,$ we denote the (centered) Hardy-Littlewood maximal operator $%
Mf $\ of $f$ by%
\begin{equation*}
Mf(x)=\underset{r>0}{\sup }\frac{1}{\left\vert B(x,r)\right\vert }%
\dint\limits_{B(x,r)}\left\vert f(y)\right\vert dy
\end{equation*}%
where the supremum is taken over all balls $B(x,r).$
\end{definition}

H\"{a}st\"{o} and Diening defined the class $A_{p(.)}$ to consist of those
weights $w$ such that%
\begin{equation*}
\left\Vert w\right\Vert _{A_{p(.)}}=\underset{B\in \text{\ss }}{\sup }%
\left\vert B\right\vert ^{-p_{B}}\left\Vert w\right\Vert
_{L^{1}(B)}\left\Vert \frac{1}{w}\right\Vert _{L^{\frac{p^{\shortmid }(.)}{%
p(.)}}(B)}<\infty ,
\end{equation*}%
where \ss\ denotes the set of all balls in $%
\mathbb{R}
^{n}$, $p_{B}=\left( \frac{1}{\left\vert B\right\vert }\int\limits_{B}\frac{1%
}{p(x)}dx\right) ^{-1}$ and $\frac{1}{p(.)}+\frac{1}{p^{\shortmid }(.)}=1$.
Note that this class is ordinary Muckenhoupt class $A_{p\left( .\right) }$
if $p\left( .\right) $ is a constant function, see \cite{Hd}.

\begin{definition}
We say that $p(.)$ satisfies the local log-H\"{o}lder continuity condition if%
\begin{equation*}
\left\vert p(x)-p(y)\right\vert \leq \frac{C}{\log \left( e+\frac{1}{%
\left\vert x-y\right\vert }\right) }
\end{equation*}%
for all $x,y\in 
\mathbb{R}
^{n}.$ If the inequality 
\begin{equation*}
\left\vert p(x)-p_{\infty }\right\vert \leq \frac{C}{\log \left(
e+\left\vert x\right\vert \right) }
\end{equation*}%
holds for some $p_{\infty }>1$, $C>0$ and all $x\in 
\mathbb{R}
^{n},$ then we say that $p(.)$ satisfies the log-H\"{o}lder decay condition.
We denote by the symbol $P^{\log }(%
\mathbb{R}
^{n})$ the class of variable exponents which are log-H\"{o}lder continuous,
i.e. which satisfy the local log-H\"{o}lder continuity condition and the
log-H\"{o}lder decay condition.
\end{definition}

Let $p\left( .\right) ,q\left( .\right) \in P^{\log }(%
\mathbb{R}
^{n}),$ $1<p^{-}\leq p^{+}<\infty $ and $1<q^{-}\leq q^{+}<\infty $. If $%
q\left( .\right) \leq p\left( .\right) $, then there exists a constant $C>0$
depending on the characteristics of $p\left( .\right) $ and $q\left(
.\right) $ such that $\left\Vert w\right\Vert _{A_{p(.)}}\leq C\left\Vert
w\right\Vert _{A_{q(.)}}$. This yields that%
\begin{equation*}
A_{1}\subset A_{p^{-}}\subset A_{p(.)}\subset A_{p^{+}}\subset A_{\infty }
\end{equation*}%
for $p\left( .\right) \in P^{\log }(%
\mathbb{R}
^{n})$ and $1<p^{-}\leq p\left( .\right) \leq p^{+}<\infty .$

Let $p\left( .\right) \in P^{\log }(%
\mathbb{R}
^{n})$ and $1<p^{-}\leq p^{+}<\infty $. Then $M:L_{w}^{p(.)}(%
\mathbb{R}
^{n})\hookrightarrow L_{w}^{p(.)}(%
\mathbb{R}
^{n})$ if and only if $w\in A_{p(.)},$ see \cite{Hd}.

We use the notation 
\begin{equation*}
\Phi \left( 
\mathbb{R}
^{n}\right) =\left\{ p(.):1<p^{-}\leq p(.)\leq p^{+}<\infty ,\text{ }%
\left\Vert Mf\right\Vert _{L_{w}^{p(.)}(%
\mathbb{R}
^{n})}\leq C\left\Vert f\right\Vert _{L_{w}^{p(.)}(%
\mathbb{R}
^{n})}\right\} \text{,}
\end{equation*}%
i.e. the maximal operator $M$ is bounded on $L_{w}^{p(.)}(%
\mathbb{R}
^{n})$. Hence we can find a sufficient condition for $p(.)\in \mathcal{P}%
\left( 
\mathbb{R}
^{n}\right) $.

\begin{proposition}
\label{proposition1}(\cite{Ayd})Let $w$ be a weight function and $%
1<p^{-}\leq p\left( .\right) \leq p^{+}<\infty $. If $w^{-\frac{1}{p(.)-1}%
}\in L_{loc}^{1}\left( 
\mathbb{R}
^{n}\right) $, then $L_{w}^{p(.)}(%
\mathbb{R}
^{n})\hookrightarrow L_{loc}^{1}\left( 
\mathbb{R}
^{n}\right) $.
\end{proposition}

\begin{definition}
Let $\varphi :%
\mathbb{R}
^{n}\longrightarrow 
\mathbb{R}
$ be a nonnegative, radial, decreasing function belonging to $C_{0}^{\infty
}\left( 
\mathbb{R}
^{n}\right) $ and having the properties

\begin{enumerate}
\item[\textit{(i)}] $\varphi (x)=0$ if $\left\vert x\right\vert \geq 1,$

\item[\textit{(ii)}] $\dint\limits_{%
\mathbb{R}
^{n}}\varphi (x)dx=1$.
\end{enumerate}

Let $\varepsilon >0$. If the function $\varphi _{\varepsilon
}(x)=\varepsilon ^{-n}\varphi (\frac{x}{\varepsilon })$ is nonnegative,
belongs to $C_{0}^{\infty }\left( 
\mathbb{R}
^{n}\right) $, and satisfies

\begin{enumerate}
\item[\textit{(i)}] $\varphi _{\varepsilon }(x)=0$ if $\left\vert
x\right\vert \geq \varepsilon $ and

\item[\textit{(ii)}] $\dint\limits_{%
\mathbb{R}
^{n}}\varphi _{\varepsilon }(x)dx=1$,
\end{enumerate}

then $\varphi _{\varepsilon }$ is called a mollifier and we define the
convolution by%
\begin{equation*}
\varphi _{\varepsilon }\ast f(x)=\dint\limits_{%
\mathbb{R}
^{n}}\varphi _{\varepsilon }(x-y)f(y)dy.
\end{equation*}
\end{definition}

The following proposition was proved in \cite[Proposition 2.7]{Duo}.

\begin{proposition}
\label{proposition2}Let $\varphi _{\varepsilon }$ be a mollifier and $f\in
L_{loc}^{1}\left( 
\mathbb{R}
^{n}\right) $. Then 
\begin{equation*}
\underset{\varepsilon >0}{\sup }\left\vert \varphi _{\varepsilon }\ast
f(x)\right\vert \leq Mf(x).
\end{equation*}
\end{proposition}

\begin{proposition}
\label{proposition3}(\cite{Ayd})Let $p\left( .\right) \in \mathcal{P}\left( 
\mathbb{R}
^{n}\right) ,$ $w\in A_{p\left( .\right) }$\ and $f\in L_{w}^{p(.)}(%
\mathbb{R}
^{n})$. Then $\varphi _{\varepsilon }\ast f\longrightarrow f$ in $%
L_{w}^{p(.)}(%
\mathbb{R}
^{n})$ as $\varepsilon \longrightarrow 0^{+}$.
\end{proposition}

As a direct consequence of Proposition \ref{proposition3} there follows.

\begin{corollary}
\label{corollary1}The class $C_{0}^{\infty }\left( 
\mathbb{R}
^{n}\right) $ denotes continuous functions having continuous derivatives of
all orders with compact support on $%
\mathbb{R}
^{n}.$ Now, let\textbf{\ }$p\left( .\right) \in \mathcal{P}\left( 
\mathbb{R}
^{n}\right) $ and $w\in A_{p\left( .\right) }$. Then $C_{0}^{\infty }\left( 
\mathbb{R}
^{n}\right) $ is dense in $L_{w}^{p(.)}(%
\mathbb{R}
^{n}).$
\end{corollary}

\begin{definition}
Let $w=\left\{ w_{k}\right\} $ be a sequence of positive numbers. The
weighted variable sequence Lebesgue spaces $l_{p_{n}}(w)$ is defined by%
\begin{equation*}
l_{p_{n}}(w)=\left\{ x=\left\{ x_{k}\right\} :\exists \lambda
>0,\tsum\limits_{k=1}^{\infty }\left( \lambda \left\vert x_{k}\right\vert
\right) ^{p_{k}}w_{k}<\infty \right\}
\end{equation*}

equipped with the norm 
\begin{equation*}
\left\Vert x\right\Vert _{l_{p_{n}}(w)}:=\left\Vert xw^{\frac{1}{p_{n}}%
}\right\Vert _{l_{p_{n}}(w)}=\inf \left\{ \lambda
>0:\tsum\limits_{k=1}^{\infty }\left( \frac{\left\vert x_{k}\right\vert }{%
\lambda }\right) ^{p_{k}}w_{k}\leq 1\right\} .
\end{equation*}
\end{definition}

The following theorem was proved for weighted variable exponent sequence
spaces by \cite{Ba} and \cite{Gom}. Also, this theorem for a constant
exponential case was obtained by \cite{Hoh}.

\begin{theorem}
\label{theorem2}Let $\tciFourier \subset l_{p_{n}}(w)$, $p^{+}<\infty $.
Then the subset $\tciFourier $ is precompact in $l_{p_{n}}(w)$ if and only if

\begin{enumerate}
\item[\textit{(i)}] $\tciFourier $ is bounded, i.e. $\forall x=\left\{
x_{k}\right\} \in \tciFourier $, $\exists C>0,$ $\tsum\limits_{k=1}^{\infty
}\left\vert x_{k}\right\vert ^{p_{k}}w_{k}\leq C$

\item[\textit{(ii)}] For every $\varepsilon >0$ there is a $K=K(\varepsilon
)>0$ such that for every $x=\left\{ x_{k}\right\} $ in $\tciFourier $ 
\begin{equation*}
\left\Vert xw^{\frac{1}{p_{k}}}\right\Vert _{l_{p_{k}}(k>K)}<\varepsilon
\end{equation*}%
or equivalently%
\begin{equation*}
\tsum\limits_{k=K+1}^{\infty }\left\vert x_{k}\right\vert
^{p_{k}}w_{k}<\varepsilon .
\end{equation*}
\end{enumerate}
\end{theorem}

The following theorem is an extension to the weighted variable exponent
Lebesgue spaces of the classical Riesz-Kolmogorov Theorem, see \cite{Ba}.

\begin{theorem}
\label{theorem3}Let $p\left( .\right) \in P^{\log }(%
\mathbb{R}
^{n})$ and $1<p^{-}\leq p^{+}<\infty $. Assume that $w$ is a weight function
and $w\in A_{p(.)}$. Then $\tciFourier \subset L_{w}^{p(.)}(%
\mathbb{R}
^{n})$ is relatively compact if and only if

\begin{enumerate}
\item[\textit{(i)}] $\tciFourier $ is bounded in $L_{w}^{p(.)}(%
\mathbb{R}
^{n})$, i.e. $\underset{f\in \tciFourier }{\sup }\left\Vert f\right\Vert
_{L_{w}^{p(.)}(%
\mathbb{R}
^{n})}<\infty $

\item[\textit{(ii)}] For every $\varepsilon >0$ there is a $\gamma >0$ such
that for all $f\in \tciFourier $ 
\begin{equation*}
\left\Vert f\right\Vert _{L_{w}^{p(.)}\left( \left\vert x\right\vert >\gamma
\right) }<\varepsilon
\end{equation*}%
or equivalently%
\begin{equation*}
\dint\limits_{\left\vert x\right\vert >\gamma }\left\vert f(x)\right\vert
^{p(x)}w(x)dx<\varepsilon .
\end{equation*}

\item[\textit{(iii)}] $\underset{\varepsilon \longrightarrow 0^{+}}{\lim }%
\left\Vert f\ast \varphi _{\varepsilon }-f\right\Vert _{L_{w}^{p(.)}(%
\mathbb{R}
^{n})}=0$ uniformly for $f\in \tciFourier $, where $\varphi _{\varepsilon }$
is a mollifier function.
\end{enumerate}
\end{theorem}

The following theorem can be proved for the spaces $L_{w}^{p(.)}(%
\mathbb{R}
^{n})$ as Theorem 3 and Theorem 4 in \cite{Gom}.

\begin{theorem}
\label{theorem4}Let $\tciFourier \subset L_{w}^{p(.)}(%
\mathbb{R}
^{n})$, $p\left( .\right) \in P^{\log }(%
\mathbb{R}
^{n})$, $1<p^{-}\leq p^{+}<\infty $ and $w\in A_{p(.)}$. Then the family $%
\tciFourier \subset L_{w}^{p(.)}(%
\mathbb{R}
^{n})$ is precompact in $L_{w}^{p(.)}(%
\mathbb{R}
^{n})$ if and only if

\begin{enumerate}
\item[\textit{(i)}] $\tciFourier $ is bounded in $L^{p(.)}(%
\mathbb{R}
^{n},w)$, i.e. $\underset{f\in \tciFourier }{\sup }\left\Vert f\right\Vert
_{L_{w}^{p(.)}(%
\mathbb{R}
^{n})}<\infty $

\item[\textit{(ii)}] For every $\varepsilon >0$ there is a $R>0$ such that
for all $f\in \tciFourier $%
\begin{equation*}
\dint\limits_{\left\vert x\right\vert >R}\left\vert f(x)\right\vert
^{p(x)}w(x)dx<\varepsilon
\end{equation*}

\item[\textit{(iii)}] For every $\varepsilon >0$ there is a $\delta >0$ such
that for all $f\in \tciFourier $ and $\forall \left\vert h\right\vert
<\delta $%
\begin{equation*}
\left\Vert f_{h}-f\right\Vert _{L_{w}^{p(.)}(%
\mathbb{R}
^{n})}<\varepsilon
\end{equation*}%
or equivalently%
\begin{equation*}
\dint\limits_{%
\mathbb{R}
^{n}}\left\vert f_{h}(x)-f(x)\right\vert ^{p(x)}w(x)dx<\varepsilon
\end{equation*}%
where $f_{h}(x)=\left( f\right) _{B\left( x,h\right) }=\frac{1}{\left\vert
B\left( x,h\right) \right\vert }\int_{B\left( x,h\right) }f(t)dt$.
\end{enumerate}
\end{theorem}

\section{\textbf{Weighted Variable Exponent Amalgam Spaces}}

\begin{definition}
The space $L_{loc,w}^{p(.)}\left( 
\mathbb{R}
^{n}\right) $ is to be space of functions on $%
\mathbb{R}
^{n}$ such that $f$ restricted to any compact subset $K$ of $%
\mathbb{R}
^{n}$ belongs to $L_{w}^{p(.)}(%
\mathbb{R}
^{n}).$ Note that the embeddings $L_{w}^{p(.)}\left( 
\mathbb{R}
^{n}\right) \hookrightarrow L_{loc,w}^{p(.)}\left( 
\mathbb{R}
^{n}\right) \hookrightarrow L_{loc}^{1}\left( 
\mathbb{R}
^{n}\right) $ hold.
\end{definition}

Let $1\leq p(.),q<\infty $ and $J_{k}=\left[ k,k+1\right) $, $k\in 
\mathbb{Z}
.$ The weighted variable exponent amalgam spaces $\left( L_{w}^{p(.)},\ell
^{q}\right) $ are defined by%
\begin{equation*}
\left( L_{w}^{p(.)},\ell ^{q}\right) =\left\{ f\in L_{loc,w}^{p(.)}\left( 
\mathbb{R}
\right) :\left\Vert f\right\Vert _{\left( L_{w}^{p(.)},\ell ^{q}\right)
}<\infty \right\} ,
\end{equation*}%
where 
\begin{equation*}
\left\Vert f\right\Vert _{\left( L_{w}^{p(.)},\ell ^{q}\right) }=\left(
\tsum\limits_{k\in 
\mathbb{Z}
}\left\Vert f\chi _{J_{k}}\right\Vert _{L_{w}^{p(.)}(%
\mathbb{R}
)}^{q}\right) ^{\frac{1}{q}}.
\end{equation*}%
It is well known that $\left( L_{w}^{p(.)},\ell ^{q}\right) $ is a Banach
space and does not depend on the particular choice of $J_{k}$, that is, $%
J_{k}$ can be equal to $\left[ k,k+1\right) $, $[k,k+1]$ or $(k,k+1)$. If
the weight $w$ is a constant \ function, then the space $\left(
L_{w}^{p(.)},\ell ^{q}\right) $ coincides with $\left( L^{p(.)},\ell
^{q}\right) $. Moreover, If the exponent $p(.)$ and the weight $w$ are
constant functions, then we have the usual amalgam space $\left( L^{p},\ell
^{q}\right) $, see \cite{Ady}, \cite{Ho}, \cite{Wie}. The dual space of $%
\left( L_{w}^{p(.)},\ell ^{q}\right) $ is isometrically isomorphic to $%
\left( L_{w^{\ast }}^{r(.)},\ell ^{t}\right) $ where $\frac{1}{p(.)}+\frac{1%
}{r(.)}=1$, $\frac{1}{q}+\frac{1}{t}=1$ and $w^{\ast }=w^{1-r\left( .\right)
}$. Also, the space $\left( L_{w}^{p(.)},\ell ^{q}\right) $ is reflexive.
Moreover, it is known that $\left( L_{w}^{p(.)},\ell ^{q}\right) $ is a
solid Banach function space, see \cite{Ag}.

In 2014, Meskhi and Zaighum \cite{MesZa} proved the boundedness of maximal
operator for weighted variable exponent amalgam spaces under some
conditions, see \cite[Theorem 3.3]{MesZa}, \cite[Theorem 3.4]{MesZa}.
Throughout this paper, we assume that $p\left( .\right) \in P^{\log }(%
\mathbb{R}
^{n})$, $1<p^{-}\leq p\left( .\right) \leq p^{+}<\infty $, $w\in A_{p(.)}$
and the maximal operator is bounded in weighted variable exponent amalgam
spaces.

\begin{remark}
Let $\frac{1}{p(.)}+\frac{1}{r(.)}=1$ and $\frac{1}{q}+\frac{1}{s}=1$. Then
there exists a constant $C>0$ such that 
\begin{equation*}
\left\Vert fg\right\Vert _{\left( L^{1},\ell ^{1}\right) }\leq C\left\Vert
f\right\Vert _{\left( L_{w}^{p(.)},\ell ^{q}\right) }\left\Vert g\right\Vert
_{\left( L_{w^{\ast }}^{r(.)},\ell ^{s}\right) }
\end{equation*}%
for $f\in \left( L_{w}^{p(.)},\ell ^{q}\right) $ and $g\in \left( L_{w^{\ast
}}^{r(.)},\ell ^{s}\right) $. Moreover, the expression 
\begin{equation*}
\left( L_{w}^{p(.)},\ell ^{q}\right) \left( L_{w^{\ast }}^{r(.)},\ell
^{s}\right) \subset \left( L^{1},\ell ^{1}\right) =L^{1}
\end{equation*}%
is satisfied.
\end{remark}

\begin{proof}
Let $f\in \left( L_{w}^{p(.)},\ell ^{q}\right) $ and $g\in \left( L_{w^{\ast
}}^{r(.)},\ell ^{s}\right) $. Using H\"{o}lder inequality for variable
exponent Lebesgue and classical sequences spaces, we have%
\begin{eqnarray*}
\left\Vert fg\right\Vert _{\left( L^{1},\ell ^{1}\right) }
&=&\tsum\limits_{k\in 
\mathbb{Z}
}\left\Vert fg\chi _{J_{k}}\right\Vert _{L^{1}(%
\mathbb{R}
)} \\
&\leq &C\tsum\limits_{k\in 
\mathbb{Z}
}\left( \left\Vert f\chi _{J_{k}}\right\Vert _{L_{w}^{p(.)}(%
\mathbb{R}
)}\left\Vert g\chi _{J_{k}}\right\Vert _{L_{w^{\ast }}^{r(.)}(%
\mathbb{R}
)}\right) \\
&\leq &C\left( \tsum\limits_{k\in 
\mathbb{Z}
}\left\Vert f\chi _{J_{k}}\right\Vert _{L_{w}^{p(.)}(%
\mathbb{R}
)}^{q}\right) ^{\frac{1}{q}}\left( \tsum\limits_{k\in 
\mathbb{Z}
}\left\Vert g\chi _{J_{k}}\right\Vert _{L_{w^{\ast }}^{r(.)}(%
\mathbb{R}
)}^{s}\right) ^{\frac{1}{s}} \\
&\leq &C\left\Vert f\right\Vert _{\left( L_{w}^{p(.)},\ell ^{q}\right)
}\left\Vert g\right\Vert _{\left( L_{w^{\ast }}^{r(.)},\ell ^{s}\right) }.
\end{eqnarray*}%
This completes the proof.
\end{proof}

\begin{definition}
(\cite{Ady},\cite{Sq})$L_{c,w}^{p(.)}\left( 
\mathbb{R}
\right) $ denotes the functions $f$ in $L_{w}^{p(.)}(%
\mathbb{R}
)$ such that supp$f\subset 
\mathbb{R}
$ is compact, that is, 
\begin{equation*}
L_{c,w}^{p(.)}\left( 
\mathbb{R}
\right) =\left\{ f\in L_{w}^{p(.)}(%
\mathbb{R}
):\text{supp}f\text{ compact}\right\} .
\end{equation*}%
Let $K\subset 
\mathbb{R}
$ be given. The cardinality of the set 
\begin{equation*}
S(K)=\left\{ J_{k}:J_{k}\cap K\neq \varnothing \right\}
\end{equation*}%
is denoted by $\left\vert S(K)\right\vert $, where $\left\{ J_{k}\right\}
_{k\in 
\mathbb{Z}
}$ is a collection of intervals.
\end{definition}

\begin{proposition}
\label{proposition4}(\cite{Ady})If $g$ belongs to $L_{c,w}^{p(.)}\left( 
\mathbb{R}
\right) $, then

\begin{enumerate}
\item[\textit{(i)}] $\left\Vert g\right\Vert _{\left( L_{w}^{p(.)},\ell
^{q}\right) }\leq \left\vert S(K)\right\vert ^{\frac{1}{q}}\left\Vert
g\right\Vert _{L_{w}^{p(.)}(%
\mathbb{R}
)}$ for $1\leq q<\infty ,$

\item[\textit{(ii)}] $\left\Vert g\right\Vert _{\left( L_{w}^{p(.)},\ell
^{\infty }\right) }\leq \left\vert S(K)\right\vert \left\Vert g\right\Vert
_{L_{w}^{p(.)}(%
\mathbb{R}
)}$ for $q=\infty $,

\item[\textit{(iii)}] $L_{c,w}^{p(.)}\left( 
\mathbb{R}
\right) \subset \left( L_{w}^{p(.)},\ell ^{q}\right) $ for $1\leq q\leq
\infty ,$
\end{enumerate}

where $K$ is the compact support of $g$.
\end{proposition}

\begin{theorem}
\label{theorem5}$L_{c,w}^{p(.)}\left( 
\mathbb{R}
\right) $ is dense subspace of $\left( L_{w}^{p(.)},\ell ^{q}\right) $ for $%
1\leq p(.),q<\infty .$
\end{theorem}

\begin{proof}
If we use the similar techniques of Theorem 7 in \cite{Ho} or Theorem 3.6 in 
\cite{Sq}, then we can prove the theorem similarly.
\end{proof}

\begin{proposition}
\label{proposition5}$C_{c}\left( 
\mathbb{R}
\right) $, which consists of continuous functions on $%
\mathbb{R}
$ whose support is compact$,$ is dense in $\left( L_{w}^{p(.)},\ell
^{q}\right) $ for $1\leq p(.),q<\infty .$
\end{proposition}

\begin{proof}
It is clear that $C_{c}\left( 
\mathbb{R}
\right) $ is included in $\left( L_{w}^{p(.)},\ell ^{q}\right) $. Let $f\in
\left( L_{w}^{p(.)},\ell ^{q}\right) $. By Theorem \ref{theorem5}, given $%
\varepsilon >0$ there exists $g\in L_{c,w}^{p(.)}\left( 
\mathbb{R}
\right) $ such that 
\begin{equation}
\left\Vert f-g\right\Vert _{\left( L_{w}^{p(.)},\ell ^{q}\right) }<\frac{%
\varepsilon }{2}.  \label{3.1}
\end{equation}

If $E$ is the compact support of $g$, then there exists $h$ in $C_{c}\left(
E\right) $ such that%
\begin{equation*}
\left\Vert g-h\right\Vert _{L_{w}^{p(.)}(E)}<\frac{\varepsilon }{2\left\vert
S(E)\right\vert ^{\frac{1}{q}}}
\end{equation*}%
since $C_{c}\left( E\right) $ is dense in $L_{w}^{p(.)}(E),$ see \cite{Ay}$.$
Hence by Proposition \ref{proposition4}, we have%
\begin{equation}
\left\Vert g-h\right\Vert _{\left( L_{w}^{p(.)},\ell ^{q}\right) }\leq
\left\vert S(E)\right\vert ^{\frac{1}{q}}\left\Vert g-h\right\Vert
_{L_{w}^{p(.)}(E)}<\frac{\varepsilon }{2}.  \label{3.3}
\end{equation}%
If we consider the (\ref{3.1}) and (\ref{3.3}), then 
\begin{eqnarray*}
\left\Vert f-h\right\Vert _{\left( L_{w}^{p(.)},\ell ^{q}\right) } &\leq
&\left\Vert f-g\right\Vert _{\left( L_{w}^{p(.)},\ell ^{q}\right)
}+\left\Vert g-h\right\Vert _{\left( L_{w}^{p(.)},\ell ^{q}\right) } \\
&<&\frac{\varepsilon }{2}+\frac{\varepsilon }{2}=\varepsilon .
\end{eqnarray*}%
This completes the proof.
\end{proof}

By the Corollary \ref{corollary1}, the proof of the following result can be
obtained similar to Proposition \ref{proposition5}.

\begin{corollary}
The class $C_{0}^{\infty }\left( 
\mathbb{R}
\right) $ is dense in $\left( L_{w}^{p(.)},\ell ^{q}\right) $ for $1\leq
p(.),q<\infty .$
\end{corollary}

Now we give the proof of the Kolmogorov-Riesz Theorem for $\tciFourier
\subset \left( L_{w}^{p(.)},\ell ^{q}\right) $.

\begin{theorem}
\label{theorem6}A subset $\tciFourier \subset \left( L_{w}^{p(.)},\ell
^{q}\right) $ is totally bounded if and only if

\begin{enumerate}
\item[\textit{(i)}] $\tciFourier $ is bounded in $\left( L_{w}^{p(.)},\ell
^{q}\right) $, i.e. $\underset{f\in \tciFourier }{\sup }\left\Vert
f\right\Vert _{\left( L_{w}^{p(.)},\ell ^{q}\right) }<\infty $

\item[\textit{(ii)}] For every $\varepsilon >0$ there is some $\gamma >0$
such that for all $f\in \tciFourier $ 
\begin{equation*}
\left\Vert f\right\Vert _{\left( L_{w}^{p(.)},\ell ^{q}\right) \left(
\left\vert x\right\vert >\gamma \right) }<\varepsilon
\end{equation*}

\item[\textit{(iii)}] $\underset{\varepsilon \longrightarrow 0^{+}}{\lim }%
\left\Vert f\ast \varphi _{\varepsilon }-f\right\Vert _{\left(
L_{w}^{p(.)},\ell ^{q}\right) }=0$ uniformly for $f\in \tciFourier $, where $%
\varphi _{\varepsilon }$ is a mollifier function.
\end{enumerate}
\end{theorem}

\begin{proof}
Assume that $\tciFourier $ is totally bounded in $\left( L_{w}^{p(.)},\ell
^{q}\right) $. Then, for every $\varepsilon >0$ there exists a finite $%
\varepsilon $-cover for the set $\tciFourier $. This implies plainly the
boundedness of $\tciFourier $, and then we get \textit{(i)}. To prove
condition \textit{(ii)}, let $\varepsilon >0$ be given. If we take the set $%
\left\{ V_{1},V_{2},..,V_{m}\right\} $ as an $\varepsilon $-cover of $%
\tciFourier $, and $h_{j}\in V_{j}$ for $j=1,...,m$, then for a $\gamma >0$
we have 
\begin{equation*}
\left\Vert h_{j}\right\Vert _{\left( L_{w}^{p(.)},\ell ^{q}\right) \left(
\left\vert x\right\vert >\gamma \right) }<\varepsilon .
\end{equation*}%
If $f\in V_{j}$, then we have $\left\Vert f-h_{j}\right\Vert _{\left(
L_{w}^{p(.)},\ell ^{q}\right) }\leq \varepsilon .$ This follows that%
\begin{eqnarray*}
\left\Vert f\right\Vert _{\left( L_{w}^{p(.)},\ell ^{q}\right) \left(
\left\vert x\right\vert >\gamma \right) } &\leq &\left\Vert
f-h_{j}\right\Vert _{\left( L_{w}^{p(.)},\ell ^{q}\right) \left( \left\vert
x\right\vert >\gamma \right) }+\left\Vert h_{j}\right\Vert _{\left(
L_{w}^{p(.)},\ell ^{q}\right) \left( \left\vert x\right\vert >\gamma \right)
} \\
&\leq &\left\Vert f-h_{j}\right\Vert _{\left( L_{w}^{p(.)},\ell ^{q}\right)
}+\left\Vert h_{j}\right\Vert _{\left( L_{w}^{p(.)},\ell ^{q}\right) \left(
\left\vert x\right\vert >\gamma \right) }<2\varepsilon .
\end{eqnarray*}%
This implies the condition \textit{(ii)}. Finally, we show the condition 
\textit{(iii)}. Let $f\in \tciFourier $ be given. Then by Theorem \ref%
{theorem5}, given $\varepsilon >0$ there exists $g\in L_{c,w}^{p(.)}\left( 
\mathbb{R}
\right) $ such that%
\begin{equation}
\left\Vert f-g\right\Vert _{\left( L_{w}^{p(.)},\ell ^{q}\right) }<\frac{%
\varepsilon }{2\max \left\{ 1,c\right\} }.  \label{3.4}
\end{equation}%
If $E$ is the compact support of $g,$ then we have%
\begin{equation}
\left\Vert g-g\ast \varphi _{\varepsilon }\right\Vert _{L_{w}^{p(.)}(E)}<%
\frac{\varepsilon }{2\left\vert S(E)\right\vert ^{\frac{1}{q}}}  \label{3.5}
\end{equation}%
by Proposition \ref{proposition3}. Also by Proposition \ref{proposition4},
we get%
\begin{equation}
\left\Vert g-g\ast \varphi _{\varepsilon }\right\Vert _{\left(
L_{w}^{p(.)},\ell ^{q}\right) }\leq \left\vert S(E)\right\vert ^{\frac{1}{q}%
}\left\Vert g-g\ast \varphi _{\varepsilon }\right\Vert _{L_{w}^{p(.)}(E)}<%
\frac{\varepsilon }{2}.  \label{3.6}
\end{equation}%
If we consider the Proposition \ref{proposition2}, the boundedness of
maximal operator, (\ref{3.4}) and (\ref{3.6}), then we have 
\begin{eqnarray*}
\left\Vert f-f\ast \varphi _{\varepsilon }\right\Vert _{\left(
L_{w}^{p(.)},\ell ^{q}\right) } &\leq &\left\Vert f-g\right\Vert _{\left(
L_{w}^{p(.)},\ell ^{q}\right) }+\left\Vert g-g\ast \varphi _{\varepsilon
}\right\Vert _{\left( L_{w}^{p(.)},\ell ^{q}\right) }+\left\Vert g\ast
\varphi _{\varepsilon }-f\ast \varphi _{\varepsilon }\right\Vert _{\left(
L_{w}^{p(.)},\ell ^{q}\right) } \\
&\leq &\max \left\{ 1,c\right\} \left\Vert f-g\right\Vert _{\left(
L_{w}^{p(.)},\ell ^{q}\right) }+\left\Vert g-g\ast \varphi _{\varepsilon
}\right\Vert _{\left( L_{w}^{p(.)},\ell ^{q}\right) } \\
&<&\frac{\varepsilon }{2}+\frac{\varepsilon }{2}=\varepsilon .
\end{eqnarray*}%
This finishes the proof of necessity. Now, we assume that $\tciFourier
\subset \left( L_{w}^{p(.)},\ell ^{q}\right) $ satisfies all three
conditions. Let $\varepsilon >0$ be given. Denote%
\begin{equation*}
\tciFourier _{\varepsilon }=\left\{ f_{\varepsilon }:f\in \tciFourier \text{%
, }f_{\varepsilon }=f\ast \varphi _{\varepsilon }\right\} .
\end{equation*}%
By Proposition \ref{proposition2}, we have%
\begin{equation*}
\left\vert f_{\varepsilon }(x)\right\vert \leq Mf(x),
\end{equation*}%
where $Mf$ is the maximal function. Since the condition \textit{(i)} hold,
we have uniformly boundedness of all functions in $\tciFourier _{\varepsilon
}$. Now, we denote%
\begin{equation*}
\tciFourier _{\varepsilon \varepsilon }=\left\{ f_{\varepsilon \varepsilon
}:f_{\varepsilon }\in \tciFourier _{\varepsilon }\text{, }f_{\varepsilon
\varepsilon }=f_{\varepsilon }\ast \varphi _{\varepsilon }\right\} .
\end{equation*}

If we consider the Proposition \ref{proposition2} and the monotonicity of
the maximal operator, then we get%
\begin{eqnarray*}
\left\vert f_{\varepsilon \varepsilon }\left( x\right) \right\vert &\leq
&\left\vert \left( f_{\varepsilon }\ast \varphi _{\varepsilon }\right)
\left( x\right) \right\vert \leq M\left( f_{\varepsilon }\right) \left(
x\right) \\
&\leq &M\left( Mf\right) \left( x\right) .
\end{eqnarray*}%
This yields%
\begin{equation*}
\underset{f_{\varepsilon }\in \tciFourier _{\varepsilon }}{\sup }\left\Vert
f_{\varepsilon \varepsilon }\right\Vert _{\left( L_{w}^{p(.)},\ell
^{q}\right) }<\infty .
\end{equation*}%
Therefore, we get that all functions in $\tciFourier _{\varepsilon
\varepsilon }$ are uniformly bounded. If we use the techniques of Theorem 11
in \cite{Raf} for the rest of proof, then we can prove the theorem similarly.
\end{proof}

The following theorem has been given us a different characterization of
precompactness in $\left( L_{w}^{p(.)},\ell ^{q}\right) $ similar to Theorem
3 and Theorem 4 in \cite{Gom}.

\begin{theorem}
The family $\tciFourier \subset \left( L_{w}^{p(.)},\ell ^{q}\right) $ is
totally bounded in $\left( L_{w}^{p(.)},\ell ^{q}\right) $ if and only if

\begin{enumerate}
\item[\textit{(i)}] $\tciFourier $ is bounded in $\left( L_{w}^{p(.)},\ell
^{q}\right) $, i.e. $\underset{f\in \tciFourier }{\sup }\left\Vert
f\right\Vert _{\left( L_{w}^{p(.)},\ell ^{q}\right) }<\infty $

\item[\textit{(ii)}] For every $\varepsilon >0$, $r\longrightarrow 0^{+}$
and for all $f\in \tciFourier $ we have%
\begin{equation*}
\left\Vert f-\left( f\right) _{B\left( .,r\right) }\right\Vert _{\left(
L_{w}^{p(.)},\ell ^{q}\right) }<\varepsilon \text{ }
\end{equation*}%
or equivalently%
\begin{equation*}
\lim_{r\longrightarrow 0^{+}}\left\Vert f-\left( f\right) _{B\left(
.,r\right) }\right\Vert _{\left( L_{w}^{p(.)},\ell ^{q}\right) }=0
\end{equation*}%
where $\left( f\right) _{B\left( x,r\right) }=\frac{1}{\left\vert B\left(
x,r\right) \right\vert }\int_{B\left( x,r\right) }f(t)dt$.

\item[\textit{(iii)}] For every $\varepsilon >0$ there is a $\gamma >0$ such
that for all $f\in \tciFourier $%
\begin{equation*}
\left\Vert f\right\Vert _{\left( L_{w}^{p(.)},\ell ^{q}\right) \left(
\left\vert x\right\vert >\gamma \right) }<\varepsilon .
\end{equation*}
\end{enumerate}
\end{theorem}

\begin{proof}
Assume that $\tciFourier \subset \left( L_{w}^{p(.)},\ell ^{q}\right) $ is
totally bounded. Then, for every $\varepsilon >0$ there exists a finite $%
\varepsilon $-cover for the set $\tciFourier $. Thus, we take $\left\{
f_{l}\right\} _{l=1,...,m}$ $\varepsilon $-cover in $\tciFourier $ such that%
\begin{equation*}
\tciFourier \subset \tbigcup\limits_{l=1}^{m}B\left( f_{l},\varepsilon
\right) .
\end{equation*}%
The totally boundedness of $\tciFourier $ implies plainly the boundedness of 
$\tciFourier .$ Hence we get \textit{(i)}. By Theorem \ref{theorem5}, given $%
\varepsilon >0$ there exists $g\in L_{c,w}^{p(.)}\left( 
\mathbb{R}
\right) $ such that%
\begin{equation}
\left\Vert f-g\right\Vert _{\left( L_{w}^{p(.)},\ell ^{q}\right) }<\frac{%
\varepsilon }{2\max \left\{ 1,c\right\} }  \label{3.7}
\end{equation}%
where $c>0.$ Let $E$ be the compact support of $g$. Now, we will show that%
\begin{equation*}
\left\Vert g-\left( g\right) _{B\left( .,r\right) }\right\Vert
_{L_{w}^{p(.)}(E)}\longrightarrow 0
\end{equation*}%
or equivalently%
\begin{equation*}
\dint\limits_{E}\left\vert g(x)-\left( g\right) _{B\left( x,r\right)
}\right\vert ^{p(x)}w(x)dx\longrightarrow 0
\end{equation*}%
as $r\longrightarrow 0^{+}.$ By the Proposition \ref{proposition1}, we have $%
L_{w}^{p(.)}\hookrightarrow L_{loc}^{1}.$ Therefore, we have $\left(
g\right) _{B\left( x,r\right) }\longrightarrow g(x)$ for $x\in E$ as $%
r\longrightarrow 0^{+}$ by the Lebesgue differentiation theorem, see \cite%
{Hei}. If we use the boundedness of the Hardy-Littlewood maximal operator $%
Mg $ for $g\in L_{w}^{p(.)}(E)$, then we get%
\begin{eqnarray*}
\left\vert g\left( x\right) -\left( g\right) _{B\left( x,r\right)
}\right\vert ^{p\left( x\right) }w\left( x\right) &\leq &2^{p^{+}-1}\left(
\left\vert g\left( x\right) \right\vert ^{p\left( x\right) }+\left\vert
\left( g\right) _{B\left( x,r\right) }\right\vert ^{p\left( x\right)
}\right) w\left( x\right) \\
&\leq &2^{p^{+}-1}\left( \left\vert g\left( x\right) \right\vert ^{p\left(
x\right) }+\left\vert M\left( g\right) \left( x\right) \right\vert ^{p\left(
x\right) }\right) w\left( x\right) \in L^{1}\left( E\right) .
\end{eqnarray*}%
By the Lebesgue dominated convergence theorem, we have%
\begin{equation*}
\dint\limits_{E}\left\vert g(x)-\left( g\right) _{B\left( x,r\right)
}\right\vert ^{p(x)}w(x)dx<\frac{\varepsilon }{2\left\vert S(E)\right\vert ^{%
\frac{1}{q}}}.
\end{equation*}%
for sufficiently small $r>0$. Also by Proposition \ref{proposition4}, we get%
\begin{equation}
\left\Vert g-\left( g\right) _{B\left( .,r\right) }\right\Vert _{\left(
L_{w}^{p(.)},\ell ^{q}\right) }\leq \left\vert S(E)\right\vert ^{\frac{1}{q}%
}\left\Vert g-\left( g\right) _{B\left( .,r\right) }\right\Vert
_{L_{w}^{p(.)}(E)}<\frac{\varepsilon }{2}.  \label{3.10}
\end{equation}%
By Proposition \ref{proposition2}, (\ref{3.7}) and (\ref{3.10}), we have 
\begin{eqnarray*}
&&\left\Vert f-\left( f\right) _{B\left( .,r\right) }\right\Vert _{\left(
L_{w}^{p(.)},\ell ^{q}\right) } \\
&\leq &\left\Vert f-g\right\Vert _{\left( L_{w}^{p(.)},\ell ^{q}\right)
}+\left\Vert g-\left( g\right) _{B\left( .,r\right) }\right\Vert _{\left(
L_{w}^{p(.)},\ell ^{q}\right) }+\left\Vert \left( g\right) _{B\left(
.,r\right) }-\left( f\right) _{B\left( .,r\right) }\right\Vert _{\left(
L_{w}^{p(.)},\ell ^{q}\right) } \\
&\leq &\left\Vert f-g\right\Vert _{\left( L_{w}^{p(.)},\ell ^{q}\right)
}+\left\Vert M\left( g-f\right) \right\Vert _{\left( L_{w}^{p(.)},\ell
^{q}\right) }+\left\Vert g-\left( g\right) _{B\left( .,r\right) }\right\Vert
_{\left( L_{w}^{p(.)},\ell ^{q}\right) } \\
&\leq &\max \left\{ 1,c\right\} \left\Vert f-g\right\Vert _{\left(
L_{w}^{p(.)},\ell ^{q}\right) }+\left\Vert g-\left( g\right) _{B\left(
.,r\right) }\right\Vert _{\left( L_{w}^{p(.)},\ell ^{q}\right) }<\varepsilon
.
\end{eqnarray*}%
This completes the proof of \textit{(ii)}. If we use similar method in
Theorem \ref{theorem6}, then we get \textit{(iii)}.

Now, we assume that the conditions \textit{(i)}, \textit{(ii)} and \textit{%
(iii)} are satisfied. Since $\left( L_{w}^{p(.)},\ell ^{q}\right) $ is a
solid Banach function space, the proof is completed by \cite[Theorem 3.1]{Gr}%
.
\end{proof}

\begin{remark}
Let $\Omega \subset 
\mathbb{R}
^{n}$ be an open set. The set $L_{loc,w}^{p(.)}\left( \Omega \right) $ is
defined by 
\begin{equation*}
L_{loc,w}^{p(.)}\left( \Omega \right) =\left\{ f:f\chi _{K}\in
L_{w}^{p(.)}\left( \Omega \right) \text{ for any compact subset }K\subset
\Omega \right\}
\end{equation*}%
with the usual identification of functions that are equal almost everywhere.
Moreover, by \cite[Lemma 2.2]{Gg} it is well known that there exists a
sequence of compact subsets $\left\{ K_{j}\right\} _{j\in 
\mathbb{N}
}$ such that 
\begin{equation*}
K_{1}\subset K_{2}\subset ...\subset K_{j}\subset ...\text{and \ \ \ }\Omega
=\tbigcup_{j\in 
\mathbb{N}
}K_{j}
\end{equation*}%
where $K_{j}=\left\{ x\in \Omega :\left\vert x\right\vert \leq j\text{ and }%
dist\left( x,\complement \Omega \right) \geq \frac{1}{j}\right\} $. Here,
the complement of $\Omega $ is denoted by $\complement \Omega .$

$L_{loc,w}^{p(.)}\left( \Omega \right) $ is equipped with topology of $%
L_{w}^{p(.)}\left( \Omega \right) $ convergence on compact subsets of $%
\Omega $. In addition, any compact subset of $\Omega $ is contained in some $%
K_{j}$, and then the space $L_{loc,w}^{p(.)}\left( \Omega \right) $ is a
topological vector space with the countable family of seminorms,%
\begin{equation*}
p_{j}(f)=\left\Vert f\chi _{K_{j}}\right\Vert _{L_{w}^{p(.)}\left( \Omega
\right) }\text{, \ }j=1,2,...
\end{equation*}%
Moreover, the space $L_{loc,w}^{p(.)}\left( \Omega \right) $ is a complete
with respect to the metric $\left( f,g\right) \longrightarrow
\dsum\limits_{j=1}^{\infty }\min \left( (2^{-j},p_{j}(f-g)\right) $. Hence
it is obtained that $L_{loc,w}^{p(.)}\left( \Omega \right) $ is a Fr\v{e}%
chet space.
\end{remark}

The following theorem is proved by \cite{Hoh} for constant exponent.

\begin{theorem}
A subset $\tciFourier \subset L_{loc,w}^{p(.)}\left( \Omega \right) $ is
totally bounded if and only if

\begin{enumerate}
\item[\textit{(i)}] For every compact $K\subset \Omega $ there is some $C>0$
such that%
\begin{equation*}
\dint\limits_{\Omega }\left\vert f_{K}(x)\right\vert ^{p(x)}w(x)dx<C,\text{
\ }f\in \tciFourier
\end{equation*}%
where $f_{K}(x)=\left\{ 
\begin{array}{c}
f(x), \\ 
0,%
\end{array}%
\right. 
\begin{array}{c}
x\in K \\ 
otherwise%
\end{array}%
.$

\item[\textit{(ii)}] For every $\varepsilon >0$ and every compact $K\subset
\Omega $ there is some $r>0$ such that%
\begin{equation*}
\left\Vert f_{K}\ast \varphi _{\varepsilon }-f_{K}\right\Vert
_{L_{w}^{p(.)}(\Omega )}<\varepsilon ,\text{ \ }f\in \tciFourier
\end{equation*}%
where $f_{K}(x)=\left\{ 
\begin{array}{c}
f(x), \\ 
0,%
\end{array}%
\right. 
\begin{array}{c}
x\in K \\ 
otherwise%
\end{array}%
.$
\end{enumerate}
\end{theorem}

\begin{proof}
The subset $\tciFourier \subset L_{loc,w}^{p(.)}\left( \Omega \right) $ is
totally bounded in $L_{loc,w}^{p(.)}\left( \Omega \right) $ if and only if $%
\tciFourier _{j}=\left\{ f_{K_{j}}:f\in \tciFourier \right\} $ is totally
bounded for every $j$, with $K_{j}$ as defined above.
\end{proof}

\begin{remark}
Let $\left\{ A_{k}\right\} _{k\in 
\mathbb{Z}
}$ be a family of Banach spaces. We define the space $\ell ^{q}\left(
A_{k}\right) $ given by%
\begin{equation*}
\ell ^{q}\left( A_{k}\right) =\left\{ x=\left( x_{k}\right) :x_{k}\in
A_{k},\left\Vert x\right\Vert <\infty \right\} ,
\end{equation*}%
where $\left\Vert x\right\Vert =\left( \tsum\limits_{k\in 
\mathbb{Z}
}\left\Vert x_{k}\right\Vert _{A_{k}}^{q}\right) ^{\frac{1}{q}}.$ It can be
seen that $\ell ^{q}\left( A_{k}\right) $ is a Banach space with respect to
the norm $\left\Vert .\right\Vert .$ Moreover, $\left( L_{w}^{p(.)},\ell
^{q}\right) $ is a particular case of $\ell ^{q}\left( A_{k}\right) $.
Indeed, if we define the amalgam space as 
\begin{equation*}
\left( L_{w}^{p(.)},\ell ^{q}\right) =\left\{ f\in L_{loc,w}^{p(.)}\left( 
\mathbb{R}
\right) :\left\{ \left\Vert f\chi _{J_{k}}\right\Vert _{L_{w}^{p(.)}(%
\mathbb{R}
)}\right\} _{k\in 
\mathbb{Z}
}\in \ell ^{q}\right\}
\end{equation*}%
take $A_{k}=L_{w}^{p(.)}\left( J_{k}\right) ,$ $J_{k}=\left[ k,k+1\right) $,
then the map $f\longrightarrow \left( f_{k}\right) $, $f_{k}=f\chi _{J_{k}}$
is an isometric isomorphism from $\left( L_{w}^{p(.)},\ell ^{q}\right) $ to $%
\ell ^{q}\left( L_{w}^{p(.)}\left( J_{k}\right) \right) ,$ see \cite{Ayd}, 
\cite{Fou}. Hence, for the totally boundedness of $\tciFourier \subset
\left( L_{w}^{p(.)},\ell ^{q}\right) $ we can use the Theorem \ref{theorem2}%
. Note that $\tciFourier $ is totally bounded in $\left( L_{w}^{p(.)},\ell
^{q}\right) $ if and only if the set $\left\{ \left\{ \left\Vert f\chi
_{J_{k}}\right\Vert _{L_{w}^{p(.)}(%
\mathbb{R}
)}\right\} _{k\in 
\mathbb{Z}
}:f\in \tciFourier \right\} $ is totally bounded in $\ell ^{q}\left(
L_{w}^{p(.)}\left( J_{k}\right) \right) $ for $k\in 
\mathbb{Z}
$ with $1\leq q<\infty .$
\end{remark}

\section{\textbf{Weighted Variable Exponent Sobolev Spaces}}

Let $\vartheta ^{-\frac{1}{p(.)-1}}\in L_{loc}^{1}\left( 
\mathbb{R}
^{n}\right) .$ Since every function in $L_{\vartheta }^{p(.)}(%
\mathbb{R}
^{n})$ has distributional derivatives by Proposition \ref{proposition1}, we
get that the weighted variable exponent Sobolev spaces $W_{\vartheta
}^{k,p(.)}\left( 
\mathbb{R}
^{n}\right) $ are well defined.

\begin{definition}
Let $1<p^{-}\leq p(x)\leq p^{+}<\infty $, $\vartheta ^{-\frac{1}{p(.)-1}}\in
L_{loc}^{1}\left( 
\mathbb{R}
^{n}\right) $ and $k\in 
\mathbb{N}
.$ We define the weighted variable Sobolev spaces $W_{\vartheta
}^{k,p(.)}\left( 
\mathbb{R}
^{n}\right) $ by%
\begin{equation*}
W_{\vartheta }^{k,p(.)}\left( 
\mathbb{R}
^{n}\right) =\left\{ f\in L_{\vartheta }^{p(.)}(%
\mathbb{R}
^{n}):D^{\alpha }f\in L_{\vartheta }^{p(.)}(%
\mathbb{R}
^{n}),\text{ }0\leq \left\vert \alpha \right\vert \leq k\right\}
\end{equation*}%
equipped with the norm 
\begin{equation*}
\left\Vert f\right\Vert _{W_{\vartheta }^{k,p(.)}\left( 
\mathbb{R}
^{n}\right) }=\dsum\limits_{0\leq \left\vert \alpha \right\vert \leq
k}\left\Vert D^{\alpha }f\right\Vert _{L_{\vartheta }^{p(.)}(%
\mathbb{R}
^{n})},
\end{equation*}%
where $\alpha \in 
\mathbb{N}
_{0}^{n}$ is a multi-index, $\left\vert \alpha \right\vert =\alpha
_{1}+\alpha _{2}+...+\alpha _{n}$ and $D^{\alpha }=\frac{\partial
^{\left\vert \alpha \right\vert }}{\partial _{x_{1}}^{\alpha
_{1}}...\partial _{x_{n}}^{\alpha _{n}}}$. Moreover, the space $W_{\vartheta
}^{k,p(.)}\left( 
\mathbb{R}
^{n}\right) $ is a reflexive Banach space.
\end{definition}

The space $W_{\vartheta }^{1,p(.)}\left( 
\mathbb{R}
^{n}\right) $ is defined by 
\begin{equation*}
W_{\vartheta }^{1,p(.)}\left( 
\mathbb{R}
^{n}\right) =\left\{ f\in L_{\vartheta }^{p(.)}(%
\mathbb{R}
^{n}):\left\vert \nabla f\right\vert \in L_{\vartheta }^{p(.)}(%
\mathbb{R}
^{n})\right\} .
\end{equation*}%
The dual space of $W_{\vartheta }^{1,p(.)}\left( 
\mathbb{R}
^{n}\right) $ is denoted by $W_{\vartheta ^{\ast }}^{-1,q(.)}\left( 
\mathbb{R}
^{n}\right) $ where $\vartheta ^{\ast }=\vartheta ^{1-q\left( .\right) }$.
The function $\varrho _{1,p(.),\vartheta }:W_{\vartheta }^{1,p(.)}\left( 
\mathbb{R}
^{n}\right) \longrightarrow \left[ 0,\infty \right) $ is defined as $\varrho
_{1,p(.),\vartheta }(f)=\varrho _{p(.),\vartheta }(f)+\varrho
_{p(.),\vartheta }(\nabla f)$ for every $f\in W_{\vartheta }^{1,p(.)}\left( 
\mathbb{R}
^{n}\right) $. Also, the norm $\left\Vert f\right\Vert _{W_{\vartheta
}^{1,p(.)}\left( 
\mathbb{R}
^{n}\right) }=\left\Vert f\right\Vert _{L_{\vartheta }^{p(.)}\left( 
\mathbb{R}
^{n}\right) }+\left\Vert \nabla f\right\Vert _{L_{\vartheta }^{p(.)}\left( 
\mathbb{R}
^{n}\right) }$ makes the space $W_{\vartheta }^{1,p(.)}\left( 
\mathbb{R}
^{n}\right) $ a Banach space, see \cite{Kor}.

\begin{theorem}
A subset $\tciFourier \subset W_{\vartheta }^{k,p(.)}\left( 
\mathbb{R}
^{n}\right) $ is totally bounded if and only if

\begin{enumerate}
\item[\textit{(i)}] $\tciFourier $ is bounded in $W_{\vartheta
}^{k,p(.)}\left( 
\mathbb{R}
^{n}\right) $, i.e. there is a $C>0$ such that for $f\in \tciFourier $ and $%
0\leq \left\vert \alpha \right\vert \leq k$%
\begin{equation*}
\int_{\Omega }\left\vert D^{\alpha }f\right\vert ^{p(x)}\vartheta (x)dx<C%
\text{.}
\end{equation*}

\item[\textit{(ii)}] For every $\varepsilon >0$ there is a $\gamma >0$ such
that for all $f\in \tciFourier $ and $0\leq \left\vert \alpha \right\vert
\leq k$%
\begin{equation*}
\left\Vert D^{\alpha }f\right\Vert _{p(.),\vartheta \left( \left\vert
x\right\vert >\gamma \right) }<\varepsilon
\end{equation*}%
or equivalently%
\begin{equation*}
\dint\limits_{\left\vert x\right\vert >\gamma }\left\vert D^{\alpha
}f(x)\right\vert ^{p(x)}\vartheta (x)dx<\varepsilon .
\end{equation*}

\item[\textit{(iii)}] $\underset{\varepsilon \longrightarrow 0^{+}}{\lim }%
\left\Vert D^{\alpha }\left( f\ast \varphi _{\varepsilon }\right) -D^{\alpha
}f\right\Vert _{L_{\vartheta }^{p(.)}(%
\mathbb{R}
^{n})}=0$ uniformly for $f\in \tciFourier $ and $0\leq \left\vert \alpha
\right\vert \leq k$ where $\varphi _{\varepsilon }$ is a mollifier function.
\end{enumerate}
\end{theorem}

\begin{proof}
Note that $\tciFourier $ is totally bounded in $W_{\vartheta
}^{k,p(.)}\left( 
\mathbb{R}
^{n}\right) $ if and only if $D^{\alpha }\left[ \tciFourier \right] =\left\{
D^{\alpha }f:f\in \tciFourier \right\} $ is totally bounded in $L_{\vartheta
}^{p(.)}\left( 
\mathbb{R}
^{n}\right) $ for every multi-index $\alpha $ with $0\leq \left\vert \alpha
\right\vert \leq k$ by Theorem \ref{theorem3}.
\end{proof}

Using \cite[Theorem 5]{Gom} we have an extension of \cite[Corollary 9]{Hoh}
to weighted variable exponent Sobolev spaces $W_{\vartheta }^{k,p(.)}\left( 
\mathbb{R}
^{n}\right) $.

\begin{theorem}
Let $\tciFourier \subset W_{\vartheta }^{k,p(.)}\left( 
\mathbb{R}
^{n}\right) $ be given. If the following conditions are satisfied

\begin{enumerate}
\item[\textit{(i)}] $\tciFourier $ is bounded in $W_{\vartheta
}^{k,p(.)}\left( 
\mathbb{R}
^{n}\right) $, i.e. there is $C>0$ such that for $f\in \tciFourier $ and $%
0\leq \left\vert \alpha \right\vert \leq k$%
\begin{equation*}
\int_{\Omega }\left\vert D^{\alpha }f\right\vert ^{p(x)}\vartheta (x)dx<C.
\end{equation*}

\item[\textit{(ii)}] for every $\varepsilon >0$ there is a $\gamma >0$ such
that for all $f\in \tciFourier $ and $0\leq \left\vert \alpha \right\vert
\leq k$%
\begin{equation*}
\left\Vert D^{\alpha }f\right\Vert _{L_{\vartheta }^{p\left( .\right)
}\left( \left\vert x\right\vert >\gamma \right) }<\varepsilon
\end{equation*}%
or equivalently%
\begin{equation*}
\dint\limits_{\left\vert x\right\vert >\gamma }\left\vert D^{\alpha
}f(x)\right\vert ^{p(x)}\vartheta (x)dx<\varepsilon .
\end{equation*}

\item[\textit{(iii)}] for every $\varepsilon >0$ there is a $\rho >0$ such
that 
\begin{equation*}
\dint\limits_{%
\mathbb{R}
^{n}}\left\vert D^{\alpha }f(x+y)-D^{\alpha }f(x)\right\vert
^{p(x)}\vartheta (x)dx<\varepsilon ,
\end{equation*}%
for $f\in \tciFourier $, $0\leq \left\vert \alpha \right\vert \leq k$ and $%
\left\vert y\right\vert <\rho ,$
\end{enumerate}

then, $\tciFourier $ is totally bounded.
\end{theorem}

\section{\textbf{Applications}}

In this section, using a compact embedding theorem for weighted variable
exponent Sobolev spaces $W_{\vartheta }^{1,p(.)}\left( 
\mathbb{R}
^{n}\right) ,$ we discuss totally bounded subsets of $L_{\vartheta }^{q(.)}(%
\mathbb{R}
^{n})$.

\begin{theorem}
\label{theorem11}(\cite{Li})Let $p(.),q(.),\vartheta _{1}$ and $\vartheta
_{2}$ satisfy hypotheses in \cite[Corollary 3.1]{Li}. Then the continuous
embedding $W_{\vartheta _{1}}^{1,p(.)}\left( 
\mathbb{R}
^{n}\right) \hookrightarrow L_{\vartheta _{2}}^{q(.)}(%
\mathbb{R}
^{n})$ is satisfied.
\end{theorem}

\begin{theorem}
Assume that hypotheses of Theorem \ref{theorem11} hold. Also, let $\vartheta
_{2}\in A_{q\left( .\right) }$ and $\tciFourier $ be a bounded subset of $%
W_{\vartheta _{1}}^{1,p(.)}\left( 
\mathbb{R}
^{n}\right) $. If, for every $\varepsilon >0,$ there is a $\gamma >0$ such
that for all $f\in \tciFourier $%
\begin{equation}
\varrho _{W_{\vartheta _{1}}^{1,p(.)}\left( \left\vert x\right\vert >\gamma
\right) }\left( f\right) =\dint\limits_{\left\vert x\right\vert >\gamma
}\left( \left\vert f(x)\right\vert ^{p(x)}+\left\vert \nabla f(x)\right\vert
^{p(x)}\right) \vartheta _{1}(x)dx<\varepsilon ,  \label{totallybounded}
\end{equation}%
then $\tciFourier $ is a totally bounded subset of $L_{\vartheta
_{2}}^{q(.)}(%
\mathbb{R}
^{n})$.
\end{theorem}

\begin{proof}
For the proof, we will show that $\tciFourier $ satisfies the hypotheses of
Theorem \ref{theorem4} with $p\left( .\right) $ replaced by $q\left(
.\right) .$ By Theorem \ref{theorem11}, there exists a $C>0$ such that%
\begin{equation}
\left\Vert f\right\Vert _{L_{\vartheta _{2}}^{q(.)}\left( 
\mathbb{R}
^{n}\right) }\leq C\left\Vert f\right\Vert _{W_{\vartheta
_{1}}^{1,p(.)}\left( 
\mathbb{R}
^{n}\right) }  \label{embedding}
\end{equation}%
for all $f\in \tciFourier \subset W_{\vartheta _{1}}^{1,p(.)}\left( 
\mathbb{R}
^{n}\right) .$ This yields the condition \textit{(i)} of Theorem \ref%
{theorem4}. Now, we set a function $u\left( .\right) =f\left( .\right) \chi
\left( \left\vert .\right\vert -\gamma \right) $ where $\chi \left(
\left\vert .\right\vert -\gamma \right) =\left\{ 
\begin{array}{c}
0,\text{ \ \ }\left\vert .\right\vert \leq \gamma \\ 
1,\text{ \ \ }\left\vert .\right\vert >\gamma%
\end{array}%
\right. .$ If we consider the (\ref{embedding}) and \cite[Proposition 2.3]%
{Li}, then we have%
\begin{equation}
\left\Vert u\right\Vert _{L_{\vartheta _{2}}^{q(.)}\left( 
\mathbb{R}
^{n}\right) }\leq C\max \left\{ \left( \varrho _{W_{\vartheta
_{1}}^{1,p(.)}\left( 
\mathbb{R}
^{n}\right) }\left( u\right) \right) ^{\frac{1}{p^{-}}},\left( \varrho
_{W_{\vartheta _{1}}^{1,p(.)}\left( 
\mathbb{R}
^{n}\right) }\left( u\right) \right) ^{\frac{1}{p^{+}}}\right\} .
\label{normmoduler}
\end{equation}%
By the expression (\ref{totallybounded}), we get%
\begin{eqnarray*}
\varrho _{W_{\vartheta _{1}}^{1,p(.)}\left( 
\mathbb{R}
^{n}\right) }\left( u\right) &=&\dint\limits_{\left\vert x\right\vert
>\gamma }\left( \left\vert u\left( x\right) \right\vert ^{p\left( x\right)
}+\left\vert \nabla u\left( x\right) \right\vert ^{p\left( x\right) }\right)
\vartheta _{1}\left( x\right) dx \\
&&+\dint\limits_{\left\vert x\right\vert \leq \gamma }\left( \left\vert
u\left( x\right) \right\vert ^{p\left( x\right) }+\left\vert \nabla u\left(
x\right) \right\vert ^{p\left( x\right) }\right) \vartheta _{1}\left(
x\right) dx \\
&=&\dint\limits_{\left\vert x\right\vert >\gamma }\left( \left\vert f\left(
x\right) \right\vert ^{p\left( x\right) }+\left\vert \nabla f\left( x\right)
\right\vert ^{p\left( x\right) }\right) \vartheta _{1}\left( x\right)
dx<\varepsilon .
\end{eqnarray*}%
Since $\varepsilon $ is sufficiently small, we obtain%
\begin{equation*}
\varrho _{L_{\vartheta _{2}}^{q(.)}\left( 
\mathbb{R}
^{n}\right) }\left( u\right) \leq \left\Vert u\right\Vert _{L_{\vartheta
_{2}}^{q(.)}\left( 
\mathbb{R}
^{n}\right) }<\varepsilon ^{\ast }.
\end{equation*}%
Therefore, we have%
\begin{eqnarray*}
\varrho _{L_{\vartheta _{2}}^{q(.)}\left( 
\mathbb{R}
^{n}\right) }\left( u\right) &=&\dint\limits_{\left\vert x\right\vert
>\gamma }\left\vert u\left( x\right) \right\vert ^{q\left( x\right)
}\vartheta _{2}\left( x\right) dx+\dint\limits_{\left\vert x\right\vert \leq
\gamma }\left\vert u\left( x\right) \right\vert ^{q\left( x\right)
}\vartheta _{2}\left( x\right) dx \\
&=&\dint\limits_{\left\vert x\right\vert >\gamma }\left\vert f\left(
x\right) \right\vert ^{q\left( x\right) }\vartheta _{2}\left( x\right)
dx<\varepsilon ^{\ast }
\end{eqnarray*}%
for all $f\in \tciFourier .$ This completes the condition of \textit{(ii) }%
of Theorem \ref{theorem4}. Now, we will consider the rest of proof. Let $%
f\in \tciFourier .$ Since the space $C_{c}\left( 
\mathbb{R}
^{n}\right) $ is dense in $L_{\vartheta _{2}}^{q(.)}\left( 
\mathbb{R}
^{n}\right) ,$ given $\varepsilon >0$ there exists $g\in C_{c}\left( 
\mathbb{R}
^{n}\right) $ such that%
\begin{equation}
\left\Vert f-g\right\Vert _{L_{\vartheta _{2}}^{q(.)}\left( 
\mathbb{R}
^{n}\right) }<\frac{\varepsilon }{2}.  \label{son1}
\end{equation}

Now, we will show that%
\begin{equation}
\left\Vert g-\left( g\right) _{B\left( .,h\right) }\right\Vert
_{L_{\vartheta _{2}}^{q(.)}\left( 
\mathbb{R}
^{n}\right) }\longrightarrow 0  \label{son2}
\end{equation}%
or equivalently%
\begin{equation*}
\dint\limits_{%
\mathbb{R}
^{n}}\left\vert g(x)-\left( g\right) _{B\left( x,h\right) }\right\vert
^{q(x)}\vartheta _{2}(x)dx\longrightarrow 0
\end{equation*}%
as $h\longrightarrow 0^{+}.$ By the Proposition \ref{proposition1}, we have $%
L_{\vartheta _{2}}^{q(.)}\hookrightarrow L_{loc}^{1}.$ Therefore, we have $%
\left( g\right) _{B\left( x,h\right) }\longrightarrow g(x)$ for $x\in 
\mathbb{R}
^{n}$ as $h\longrightarrow 0^{+}$ by the Lebesgue differentiation theorem,
see \cite{Hei}. If we use the boundedness of the Hardy-Littlewood maximal
operator $Mg$ for $g\in L_{\vartheta _{2}}^{q(.)}\left( 
\mathbb{R}
^{n}\right) $, then we get%
\begin{eqnarray*}
\left\vert g\left( x\right) -\left( g\right) _{B\left( x,h\right)
}\right\vert ^{q\left( x\right) }\vartheta _{2}\left( x\right) &\leq
&2^{q^{+}-1}\left( \left\vert g\left( x\right) \right\vert ^{q\left(
x\right) }+\left\vert \left( g\right) _{B\left( x,h\right) }\right\vert
^{q\left( x\right) }\right) \vartheta _{2}\left( x\right) \\
&\leq &2^{q^{+}-1}\left( \left\vert g\left( x\right) \right\vert ^{q\left(
x\right) }+\left\vert M\left( g\right) \left( x\right) \right\vert ^{q\left(
x\right) }\right) \vartheta _{2}\left( x\right) \in L^{1}\left( 
\mathbb{R}
^{n}\right) .
\end{eqnarray*}%
By the Lebesgue dominated convergence theorem, we have%
\begin{equation*}
\dint\limits_{%
\mathbb{R}
^{n}}\left\vert g(x)-\left( g\right) _{B\left( x,h\right) }\right\vert
^{q(x)}\vartheta _{2}(x)dx<\frac{\varepsilon }{2}.
\end{equation*}%
for sufficiently small $h>0$. Therefore, if we consider (\ref{son1}) and (%
\ref{son2}), then we obtain%
\begin{eqnarray*}
&&\left\Vert f-\left( f\right) _{B\left( .,h\right) }\right\Vert
_{L_{\vartheta _{2}}^{q(.)}\left( 
\mathbb{R}
^{n}\right) } \\
&\leq &\left\Vert f-g\right\Vert _{L_{\vartheta _{2}}^{q(.)}\left( 
\mathbb{R}
^{n}\right) }+\left\Vert g-\left( g\right) _{B\left( .,h\right) }\right\Vert
_{L_{\vartheta _{2}}^{q(.)}\left( 
\mathbb{R}
^{n}\right) }+\left\Vert \left( g\right) _{B\left( .,h\right) }-\left(
f\right) _{B\left( .,h\right) }\right\Vert _{L_{\vartheta _{2}}^{q(.)}\left( 
\mathbb{R}
^{n}\right) } \\
&\leq &\left\Vert f-g\right\Vert _{L_{\vartheta _{2}}^{q(.)}\left( 
\mathbb{R}
^{n}\right) }+\left\Vert M\left( g-f\right) \right\Vert _{L_{\vartheta
_{2}}^{q(.)}\left( 
\mathbb{R}
^{n}\right) }+\left\Vert g-\left( g\right) _{B\left( .,h\right) }\right\Vert
_{L_{\vartheta _{2}}^{q(.)}\left( 
\mathbb{R}
^{n}\right) }<\varepsilon
\end{eqnarray*}%
This completes the proof.
\end{proof}

\end{document}